\documentclass{amsart}
\usepackage{amssymb,amsthm,amstext,amsmath,bm}
\usepackage[inline]{enumitem}

\usepackage{hyperref} 

\usepackage{xcolor} 

\newcommand{\N}{\ensuremath{\mathbb{N}}}
\newcommand{\Q}{\ensuremath{\mathbb{Q}}}
\newcommand{\E}{\mathfrak{E}}
\newcommand{\Ec}{\mathfrak{E}_c}

\newtheoremstyle{theorem}
     {11pt}
     {11pt}
     {}
     {}
     {\bfseries}
     {}
     {.5em}
     {\noindent\thmnumber{#2}. \thmname{#1}{\rm\thmnote{#3}}}

\theoremstyle{theorem}

\newtheorem{defi}{Definition}[section]
\newtheorem{teo}[defi]{Theorem}
\newtheorem{prop}[defi]{Proposition}
\newtheorem{lema}[defi]{Lemma}
\newtheorem{cor}[defi]{Corollary}
\newtheorem{pre}[defi]{Question}

\newtheorem{ex}[defi]{Example}

\title[]{Hyperspaces of dimension 1}

\author[A. Zaragoza]{Alfredo Zaragoza}

\address[A. Zaragoza]{Departamento de Matemáticas, Facultad de Ciencias, Universidad Nacional Autónoma de México, Circuito Exterior s/n, Ciudad Universitaria, Coyoacán, 04510, Mexico city, Mexico
}
\email[A. Zaragoza]{soad151192@icloud.com}
\thanks{This work is part of the doctoral work of the author at UNAM, Mexico city, under the direction of the Hernández-Gutiérrez. This research was supported by a CONACyT doctoral scholarship with number 696239.}


\begin{document}

\begin{abstract}
In a previuos paper the author asked if there exists a one-dimensional space $X$ that is not almost zero-dimensional, such that the dimension of the hyperspace of compact subsets of $X$ is one-dimensional.
In this short note we give examples of spaces $X$ that are not almost zero-dimensional such that $X$ is one-dimensional and their hyperspace of compacta of $X$ also is one-dimensional. 
\end{abstract}

\maketitle

\section{Introduction}

All spaces will be assumed to be separable and metrizable.
A space $X$ is zero-dimensional\index{zero-dimensional space} if it has a base of clopen sets. A space $X$ is one-dimensional if and only if it has a base $\beta$ of neighborhoods such that $ bd_X (U)$ and is zero-dimensional and nonempty for any $U\in \beta$. If $ X $ has dimension one we write $dim(X)=1$. In general we can define the dimension of a space $ X $ for any $n\in \N$ (see \cite{vm-inf_dim_funct_spaces}) but in this work we will only use the definition of dimension 0 and 1.
 For a space $X$, $\mathcal{K}(X)$ denotes the hyperspace of non-empty compact subsets of $X$ with the
Vietoris topology; for any $n\in  \N$, $\mathcal{F}_n(X)$ 
is the subspace of $\mathcal{K}(X)$ consisting
of all the non-empty subsets that have cardinality less or equal to $n$; and $\mathcal{F}(X)$ is the subspace of $\mathcal{K}(X)$ of finite subsets of $X$. 
For $n\in \N$ and subsets $U_1,\ldots, U_n$ of a topological space $X$, we denote by $\langle  U_{1},\ldots ,U_{n}\rangle$ the collection $\left\lbrace   F \in \mathcal{K}(X):F\subset \bigcup_{k=1}^n U_k, F\cap U_{k}\neq \emptyset \textit{ for } k \leq n \right\rbrace $. Recall that the Vietoris topology on $\mathcal{K}(X)$ has as its canonical base all the sets of the form $\langle  U_{1},\ldots ,U_{n}\rangle$ where $U_k$ is a non-empty open subset of $X$ for each $k\leq n$.
Note that if $X$ is a separable metrizable space, then they every subspace of $\mathcal{K}(X)$ is also a separable metrizable space (see \cite[Theorem 3.3 and Propositions 4.4 and 4.5.2]{Sub}). In \cite{zaragoza-1} it was shown that if $X$ is a almost-zero dimensional space, then $dim(X)=dim( \mathcal{K}(X))$.
We are going to show spaces $X$ of one dimension that are not almost zero dimensional such that $dim(\mathcal{K}(X))=1$. The main results of this work are:

\begin{teo}\label{pol}
There exists a connected space $ X $ such that $ dim (X) = dim (\mathcal{K} (X)) = 1 $.
\end{teo}
\begin{teo}\label{dijstra}
There exists a totally disconnected space $X$ which is not AZD such that $ dim (X) = dim (\mathcal{K} (X)) = 1 $.
\end{teo}

The Theorem \ref{pol} was suggested in a letter by Roman Pol.

\section{Preliminaries}

 A space $(X,\mathcal{T})$ is almost zero-dimensional(AZD) if
there is a zero-dimensional topology $\mathcal{W}$ in $X$ such that $\mathcal{W}$ is coarser than $\mathcal{T}$ and has the property
that every point in $X$ has a local neighborhood base consisting of sets that are closed with respect
to $\mathcal{W}$.
 This concept was
introduced by Oversteegen and Tymchatyn in \cite{k}. They proved that almost zero-dimensional
spaces are at most 1-dimensional.
Recall that Erd\H{o}s space is defined as:
$$ \mathfrak{E} = \{(x_n)_{n\in \omega} \in  \ell^2 : x_i \in  \Q , \textit{for all } i\in \omega \};$$
and complete Erd\H{o}s space as 
$$ \mathfrak{E}_{\mathrm{c}} = \{(x_n)_{n\in \omega} \in  \ell^2 : x_i \in  \{0\}\cup \{1/n: n\in \N\} \textit{ for all } i\in \omega \}$$
when $\ell^2$ is the Hilbert space of all square summable real sequences. 
It's known that Erd\H{o}s space, and complete Erd\H{o}s space are almost zero-dimensional spaces
which are not zero-dimensional spaces(see \cite{ME}).
 A space $X$ is called cohesive if every point of the space has a neighborhood that does not contain nonempty proper clopen subsets of $X$.
 
\begin{lema}[{\cite{er}}]\label{1cohesion}

$\E$ and $\Ec$ are cohesive spaces.
\end{lema}

A one-point connectification of a space $X$ is a connected extension $Y$ of the
space such that the remainder $Y \setminus X$ is a singleton.

\begin{ex}\label{ejemdeecp}
Let $ p $ be a point outside $ \Ec $, consider $\Ec^+=\Ec\cup\{p\}$ whose neighbourhoods of $\{p\}$ are the complements of closed bounded sets of $\Ec$. Then $\Ec^+$ is metric separable connected space.

\end{ex}

It is known that if a space admits a one-point connectification, then it is cohesive. Moreover if an almost zero-dimensional space is cohesive, then it admits a one point connectification (see \cite[Proposition 5.4, p. 22]{ME}).

Let $ X $ be an $ AZD $ and cohesive space (for example $\E$, $\Ec$), then $ X $ has a one-point connectification. Suppose that $ Y =\{p\}\cup X$ where $p\notin X$. Since $ Y $ is connected then $ Y $ is not an AZD space.

Now let us consider $\Ec^+$ of  example \ref{ejemdeecp} and $N=\{0\}\cup \{1/n:n\in \N\}$. Let  $$P=[\Ec\times \{1/n:n\in \N\}]\cup (p,0)$$ with the topology inherited from $\Ec^+\times N$, then is a totally disconnected and is not an AZD space (see \cite[Example 3.6]{sums}).
The spaces $Y$ and $P$ are the spaces we will use to prove the main result.

Another important result for proving the main Theorems is the following: 
\begin{prop}\label{HZD}\cite[Proposition 2.2]{zaragoza-1}
$X$ is an $AZD$ space if only if $\mathcal{K}(X)$ is an $AZD$ space.

\end{prop}

\section{Proof of main Theorems}

Let $Z\in \{P, Y\}$ and $d$ a metric for $Z$. For each $n\in \N$, let $B_n=\{z\in Z: d(z, q)<1/n\}$ and let $E_n=Z\setminus B_n$, where $q=p$ if $Z=Y$ or $q=(p,0)$ if $Z=P$.
Note that for any $n\in \N$, $E_n$ is an $AZD$ space and, by Proposition \ref{HZD}, $\mathcal{K}(E_n)$ is an $AZD$ space. Let $\mathcal{N}=\prod_{n\in \N} [\mathcal{K}(E_n)\cup \{\emptyset\}]$, then $\mathcal{N}$ is an $AZD$ space 
(let's consider the set $\{\emptyset \}$ as an isolated point of $\mathcal{K}(E_n)\cup \{\emptyset\}$).

Let 
$$\mathcal{L} = \{ (K_1, K_2, \ldots) \in \mathcal{N} : \textit{for } m \geq n,  
K_m \cap E_n = K_n \}, and $$
$$\mathcal{S}=\{H\in \mathcal{K}(Z): q\in H\}$$ 
Let's consider the following functions $\mathcal{G}:\mathcal{L}\to \mathcal{S} $ and $\mathcal{G}_n: \mathcal{S}\to \pi_n[\mathcal{L}]$ (where $\pi_n$ is the projection to the $n$-th coordinate) such that  
 $$\mathcal{G}(K_1, K_2, \ldots)=\{q\} \cup \bigcup_{n\in \N} K_n, and $$
$$\mathcal{G}_n(K)=K\cap E_n$$

\begin{lema}\label{teorp}
$\mathcal{G}$ is well defined and is a homeomorphism.
\end{lema}

\begin{proof}
To prove that $\mathcal{G}$ is well defined, let $\mathcal{U}$ be an open cover of $\{q\} \cup \bigcup_{n\in \N} K_n$ in $Z$. Since $q\in \{q\} \cup \bigcup_{n\in \omega} K_n$, we can suppose that $B_m\in \mathcal{U}$ for some $m$. Note that $(\{q\} \cup \bigcup_{n\in \N} K_n)\setminus K_m\subset B_m$, since $Z \setminus K_m\subset B_m$. As $\{U\cap E_m: U\in \mathcal{U}\}$ is an open cover of $K_m$ and $K_m$ is a compact subset of $E_m$, then there exists $U_1,\ldots, U_k$, such that $K_m \subset \bigcup_{i\leq k}(U_i\cap E_m)$. Therefore $\{B_m, U_1,\ldots, U_k\}$ is a finte subcover of $\mathcal{U}$, that is $\{q\} \cup \bigcup_{n\in \N} K_n$ is a compact subset of $Z$. Therefore $\mathcal{G}$ is well defined.

Let's prove that $\mathcal{G}$ is injective, let $\widehat{K}=(K_1, K_2, \ldots), \widehat{H}=(H_1, H_2, \ldots) \in \mathcal{L}$ such that $\widehat{H}\neq \widehat{ K}$, then there exists $k\in \mathbb{N}$ so that $H_k\neq K_k$. Therefore there exists $x\in H_k\setminus K_k$, then $x \in \mathcal{G}(\widehat{H})$ and $x\notin \mathcal{G}(\widehat{K})$. That is, $\mathcal{G}$ is injective.
Now let's see that $\mathcal{G}$ is surjective, let $K\in \mathcal{S}$. We define $K_n=K\cap E_n$, since $E_n$ is a closed subset in $Z$, then $K_n$ is empty or is a compact subset of $E_n$. Then $K_n\in \mathcal{K}(E_n) \cup \{\emptyset\}$ for each $n\in \N$, therefore $(K\cap E_1,\ldots)\in \mathcal{L}$ and $\mathcal{G}((K\cap E_1,\ldots ))=K$. That is, $\mathcal{G}$ is surjective.
Before proving that $\mathcal{G}$ is a homeomorphism, let's show that $$\mathcal{B}=\{\langle U_1,\ldots, U_n, B_k\rangle\cap \mathcal{S} :n, k\in \N \textit{ and } U_1,\ldots, U_n
 \textit{ are } $$ $$\textit {open subsets of }Z \setminus \{ p\}  \} \cup \{\langle B_k \rangle:k\in \N\}$$ 
is a basis for $\mathcal{S}$.

Let $K\in \mathcal{S}$ and $\mathcal{W}=\langle W_1,\ldots, W_n\rangle$ an open subset of $\mathcal{K}(Z)$ such that $K\in\mathcal{W}.$
If $K_n\neq\emptyset$ for some $n \in N$, then $H_r\neq\emptyset$ for $r\geq n$, without loss of generality we can assume that $n=1$. As $p\in K\in \mathcal{W}$, then there exist $j\leq n$  such that $p\in \bigcap \{W_j: j\leq n,  p\in W_j\}$, and $k\in \N$ such that $p\in B_k \subset \bigcap \{W_j: j\leq n,  p\in W_j\}$. To find an element $\mathcal{V }$ of the base $\mathcal{B}$ such that $K\in \mathcal{V}\subset \mathcal{W}$, we consider two cases. If $K\setminus B_k= \emptyset$ or if $K\setminus B_k\neq \emptyset$. If $K\setminus B_k= \emptyset$ then $K\subset B_k$. Therefore $K\in \langle B_k \rangle \subset \mathcal{W}$. If $K\setminus B_k\neq \emptyset$ then for each $x\in K\setminus B_k$ there exist $U_x$ such that $x\in U_x\subset \bigcap \{W_j: j\leq n, x\in W_j\}$, as $K\setminus B_k $ is compact and $\{U_x:x\in K\setminus B_k\}$ is an open cover of $K\setminus B_k $, there exist $x_1, \ldots, x_l$ such that $K\setminus B_k \in \langle U_{x_1},\ldots, U_{x_l}\rangle$. Let $\mathcal{V}=\langle U_{x_1},\ldots, U_{x_l}, B_k\rangle \cap\mathcal{S}$, note that $K\in \mathcal{V}$, and $\mathcal{V}\subset \mathcal{W}\cap \mathcal{S}$.
On the other hand if $ K=\{p\}$, there exist $k\in \N$ such that $p\in B_k$ and $p\in B_k \subset \bigcap \{W_j: j\leq n,  p\in W_j\}$ this implies that $K\in \langle B_k\rangle  \subset \mathcal{W}$. Therefore $\mathcal{B}$ is a basis for $\mathcal{S}$.

Let $K=(H_1,\ldots, H_n,\ldots)\in \mathcal{L}$, and $\mathcal{U}\in \mathcal{B}$ such that $H= \mathcal{G} (K)\in \mathcal{U}$.
If $H= \{p \}$, then $\mathcal{U}=\langle B_k\rangle$ for some $k\in \N$ and $H_n=\emptyset$ for each $n\in \N$. Let $W=\{ \emptyset\}^{k}\times [\langle B_k\rangle \cap ( \mathcal{K}(E_{k+1}) \cup \{\emptyset\})] \times \prod_{m > k+1}[ \mathcal{K}(E_{m}) \cup \{\emptyset\}]$, note that $K\in W$, and $\mathcal{G}[W]\subset \mathcal{U}$. 
If $H_i\neq\emptyset$ for some $i \in \N$, then $H_r\neq \emptyset$ for $r\geq n$, without loss of generality we can assume that $i=1$, and that $\mathcal{U}=\langle U_1,\ldots, U_n, B_k\rangle$ for some $n, k\in \N$ 
. 
Let $A=\{j\in \N: U_l\cap H_j \neq \emptyset \textit{ for all } l\leq n\}$ and as $\{H_k:k\in \N\}$, is not finite, then $A\neq \emptyset$. Let $r=\min A$, if $r<k$, then $F_k\cap U_j=\emptyset$ for some $j\leq n$, so $F_r\cap U_j\neq \emptyset$ and $F_r\cap U_j \subset B_k $. Let

$$N= \{(F_1, F_2, \ldots)\in \mathcal{L}: F_r\in \langle U_1,\ldots , U_n, B_k\rangle\}.  $$
Note that $K\in N$, if $F=(F_1, F_2, \ldots)\in N$  and $\mathcal{G}(F_1, F_2, \ldots)= F$, then $p\in F\setminus F_k\subset B_r\subset B_k$, so $F\in \mathcal{U}$.
If $r\leq k$, then $H_k\setminus H_r \subset \bigcup_{j\leq n}U_j $ and $H\setminus H_k \subset B_k$. Let

$$N= \{(F_1, F_2, \ldots)\in \mathcal{L}: F_k\in \langle U_1,\ldots , U_n\rangle\}.  $$
Note that $K\in N$. If $(F_1,\ldots F_k,\ldots)\in N $, and $\mathcal{G}(F_1,\ldots F_k,\ldots)=F$, then $p\in F\setminus F_k\subset B_k$, so $F\in \mathcal{U}$.
This implies that $\mathcal{G}$ is a continuous function.
\\
Finally we will show that $\mathcal{G}^{-1}$ is a continuous function, if $U$ is a basic open subset of $\mathcal{L}$, then $U=(\bigcap_{j\in F} \pi_j^{\leftarrow} [W_j])\cap \mathcal{L}$, where $W_j$ is an open subset of $\mathcal{K}(E_j)\cup\{\emptyset\}$ and $F$ is a finite subset of $\N$.
 Hence
$$(\mathcal{G}^{-1})^{\leftarrow}[U]= \bigcap_{j\in F}(\mathcal{G}^{-1})^{\leftarrow}[\pi_j^{\leftarrow} [W_j]]\cap \mathcal{S}=\bigcap_{j\in F}\mathcal{G}_n^{\leftarrow}[W_j]. $$
So that is enough to show the continuity of $\mathcal{G}_n$ for any $n$. To prove that $\mathcal{G}_n$ is continuous, it is sufficient to show that $\mathcal{G}_n^{\leftarrow}[\langle U_1,\ldots  U_k\rangle\cap \mathcal{K}(E_n)]$ and $\mathcal{G}_n^{\leftarrow}[\{\emptyset\}]$ are open subsets of $\mathcal{S}$, where $U_1,\ldots ,U_k$ are open subsets of $Z\setminus \{p\}$ such that $E_n\cap U_j\neq \emptyset$ for each $j\leq k$.
We will show that $$\mathcal{G}_n^{\leftarrow}[\langle U_1,\ldots  U_k\rangle\cap \mathcal{K}(E_n)]=\mathcal{S}\cap \langle U_1,\ldots U_k, B_n\rangle $$ and that 
$$\mathcal{G}_n^{\leftarrow}[\{\emptyset\}]=\mathcal{S}\cap \langle B_n\rangle$$

Let $H\in \mathcal{S}\cap \langle U_1,\ldots U_k, B_n\rangle $, then $H_n\neq\emptyset$, $H\setminus H_n \subset B_n$, and $H_n\in[\langle U_1,\ldots$ $  U_k\rangle \cup \{\emptyset\}]\cap \mathcal{K}(E_n) $, thus $H\in \mathcal{G}_n^{\leftarrow}[\langle U_1,\ldots  U_k\rangle\cap \mathcal{K}(E_n)]$. Let $F \in \mathcal{G}_n^{\leftarrow}[\langle U_1,\ldots  U_k\rangle\cap \mathcal{K}(E_n)]$, then $F_n=F\cap E_n\in \langle U_1,\ldots U_k\rangle\cap \mathcal{K}(E_n)$ and $F\setminus F_n\subset B_n$ then $F\in \langle U_1 \ldots ,U_k , B_n\rangle$. Let $H\in \mathcal{S}\cap \langle B_n\rangle $, then $H_n
=\emptyset$, therefore, $H_n\in\{ \emptyset\} $, thus $H\in \mathcal{G}_n^{\leftarrow}[\{ \emptyset\}]$. Let $F \in \mathcal{G}_n^{\leftarrow}[\{\emptyset\}]$, then $F_n=F\cap E_n= \emptyset$, thus $F \subset B_n$ then $F\in \langle B_n\rangle$.
This implies that $\mathcal{G}_n$ is a continuous function.
Therefore $\mathcal{G}$ is a homeomorphism.

\end{proof}

\begin{teo}\label{dimendiondeZ}
$dim(Z)=dim(\mathcal{K}(Z))=1$ 
\end{teo}

\begin{proof}
Note that $\mathcal{K}(Z)=\mathcal{K}(Z\setminus \{q\})\cup \mathcal{S}$. As $\mathcal{K}(Z\setminus \{q\})$ is an AZD cohesive space, then $dim(\mathcal{K}(Z\setminus \{q\}))=1$. By Theorem \ref{teorp}, $S$ is homeomorphic to $\mathcal{L}$ and $dim(\mathcal{L})=1$ because $\mathcal{L}$ is an AZD, but $\mathcal{L}$ is not zero dimensinal space. This implies that $dim(\mathcal{S})=1$. Thus $dim(\mathcal{K}(Z))=1$.
\end{proof}

\textbf{Proof of Theorem \ref{pol}}
\begin{proof}
Let $Y$ be a one-point connectification of $\E$. By Theorem \ref{dimendiondeZ} we have the result.
\end{proof}
\begin{cor}\label{dijstra1}
Let $X$ be an cohesive and AZD spce. If $Y$ is a one-point connectification of $X$, then  $ dim (Y) = dim (\mathcal{K} (Y)) = 1 $.
\end{cor}
\begin{proof}
By Theorem \ref{dimendiondeZ} we have the result.
\end{proof}

\textbf{Proof of Theorem \ref{dijstra}}
\begin{proof}
Consider the space $P$. By Theorem \ref{dimendiondeZ} we have the result.
\end{proof}

Note that the spaces given in the Corollaries \ref{pol} and \ref{dijstra} are unions of $ AZD $ spaces.
  A natural question is:
  \\
  Does every space $ Z $ of dimension 1 that is not $ AZD $ and  is a finite union of subspaces $ AZD $ satisfy
  that $ dim (\mathcal{K} (Z)) = 1 $ ?
 The answer to this question is negative because $ [0,1] $ is not an $ AZD $ space,
  but is a union of $ \Q \cap [0,1] $ and $ \mathbb{P} \cap [0,1] $ which are $ AZD $ spaces, and $ dim (\mathcal{K} ([0,1])) $ is not 1.
  
On the other hand it is known that if $X$ is a compact space of dimension 1, then the dimension of $\mathcal{K}(X)$ is not finite (see \cite[pag 123]{dimension1}). This implies that if a space $X$ has a compact subset of dimension 1 then the dimension of $\mathcal{K}(X)$ is not finite. Then for $\mathcal{K}(X)$ to have dimension 1 each $A\in \mathcal{K}(X)$ must have dimension zero.
With the following Theorem, we will show that it is not enough that the compact subsets of a space $X$ of dimension 1 have dimension 0 for that hyperspace of compact subsets of $X$ have to dimension 1.

\begin{teo}\cite[Theorem 4.1]{pol}\label{poll}
There exists a space $X$ of dimension 1 such that all its compacta have dimension 0 and $dim (X^2)=2$.
\end{teo}
\begin{ex}
Let $Y=X\times \{0,1\}$ where $X$ is as in Theorem \ref{poll}, then $dim(Y)=1$ and for each compact subset $F$ of $X$ we have that $dim(F)=0$. 
Let $f: X^2 \to \mathcal{K}(X)$ given by $f(x,y)=\{(x,0), (y,1)\}$. Note that $f$ is an embedding. This implies that $dim(\mathcal{K}(Y))\geq 2$.
\end{ex}
\begin{pre}
Let $X$ be a space of dimension 1, such that $dim (X^\omega)=1$ and for each $A\in \mathcal{K}(X)$, $dim (A)=0$.
Does $\mathcal{K}(X)$ have dimension 1?.
\end{pre}

\section{Acknowledgement}
   I would like to to thank Roman Pol for useful suggestions concerning the topic of this paper.

\end{document}